\documentclass[oneside, 12pt]{amsart}
\usepackage[utf8x]{inputenc}
\usepackage{amssymb, mathrsfs, amsthm, amsmath, url}
\usepackage{stmaryrd}
\usepackage[all]{xy}
\usepackage{enumerate}
\usepackage{mdwlist}
\usepackage{graphicx}
\usepackage{hyperref}
\usepackage[multiple]{footmisc}

\renewcommand{\AA}{\mathbb{A}}

\newcommand{\C}{\mathcal{C}}

\newcommand{\GG}{\mathbb{G}}

\newcommand{\m}{\mathfrak{m}}

\newcommand{\OO}{\mathcal{O}}
\newcommand{\p}{\mathfrak{p}}

\newcommand{\ZZ}{\mathbb{Z}}

\newcommand{\id}{\mathrm{id}}

\newcommand{\Ab}{\mathrm{Ab}}
\newcommand{\Set}{\mathrm{Set}}

\newcommand{\Aff}{\mathrm{Aff}}
\newcommand{\Sch}{\mathrm{Sch}}

\newcommand{\PreShv}{\mathrm{PreShv}}
\newcommand{\Shv}{\mathrm{Shv}}

\newcommand{\Fib}{\mathrm{Fib}}
\newcommand{\cS}{\mathcal{S}}

\newcommand{\Zar}{\mathrm{Zar}}
\newcommand{\Nis}{\mathrm{Nis}}
\newcommand{\Et}{\mathrm{Et}}
\newcommand{\qfh}{\mathrm{qfh}}
\newcommand{\cl}{\mathrm{cl}}
\newcommand{\fin}{\mathrm{f}}
\newcommand{\rh}{\mathrm{rh}}
\newcommand{\cdh}{\mathrm{cdh}}
\newcommand{\cdf}{\mathrm{cdf}}
\newcommand{\cdp}{\mathrm{cdp}}
\newcommand{\ldh}{{l\mathrm{dh}}}
\newcommand{\proptop}{\mathrm{prop}}
\newcommand{\eh}{\mathrm{eh}}
\newcommand{\htop}{\mathrm{h}}
\newcommand{\fppf}{\mathrm{fppf}}
\newcommand{\fps}{\mathrm{fps}}
\newcommand{\fpsl}{\mathrm{fps}l'}
\newcommand{\aff}{\ensuremath{^\mathrm{aff}}}

\newcommand{\aic}{{a.i.c.}}
\newcommand{\red}{\mathrm{red}}

\newtheorem{theo}{Theorem}[section]
\newtheorem*{theoUn}{Theorem}

\newtheorem{lemm}[theo]{Lemma}
\newtheorem{prop}[theo]{Proposition}

\theoremstyle{definition}

\newtheorem{defi}[theo]{Definition}
\newtheorem{warn}[theo]{Warning}
\newtheorem{rema}[theo]{Remark}

\DeclareMathOperator{\Spec}{Spec}
\DeclareMathOperator{\Pro}{Pro}
\DeclareMathOperator{\Frac}{Frac}

\title{Points in algebraic geometry}

\author{Ofer Gabber}%
\address{Ofer Gabber, IH\'ES, 35 route de Chartres, 
91440 Bures-sur-Yvette, France}
\email{gabber@ihes.fr}

\author{Shane Kelly}%
\address{Shane Kelly, Interactive Research Center of Science, Graduate School of Science and Engineering, Tokyo Institute of Technology,
2-12-1 Ookayama, Meguro,
Tokyo 152-8551 JAPAN}
\email{shanekelly64@gmail.com}

\begin{document}

\maketitle

\begin{abstract}
We give scheme-theoretic descriptions of the category of fibre functors on the categories of sheaves associated to the Zariski, Nisnevich, étale, rh, cdh, ldh, eh, qfh, and h topologies on the category of separated schemes of finite type over a separated noetherian base. Combined with a theorem of Deligne on the existence of enough points, this provides an algebro-geometric description of a conservative family of fibre functors on these categories of sheaves. As an example of an application we show direct image along a closed immersion is exact for all these topologies except qfh. The methods are transportable to other categories of sheaves as well.
\end{abstract}


\section*{Introduction}

Stalks play an important rôle in the theory of sheaves on a topological space, principally via the fact that a morphism of sheaves is an isomorphism if and only it induces an isomorphism on every stalk. In general topos theory, this is no longer true (see \cite[IV.7.4]{SGA41} for an example of a non-empty topos with no fibre functors). However there is an abstract topos theoretic theorem of Deligne which says that under some finiteness hypotheses which are almost always satisfied in algebraic geometry, one can indeed detect isomorphisms using fibre functors (Theorem~\ref{theo:deligne}).

For the étale topology (on the category of finite type étale morphisms over a noetherian scheme $X$) one has an extremely useful algebraic description of the fibre functors as a certain class of morphisms $\Spec(R) \to X$ where $R$ is a strictly henselian local ring. One can define a ring $R$ to be a strictly henselian local ring if for every jointly surjective family of étale morphism of finite type $\{U_i \to \Spec(R)\}_{i \in I}$, there exists $i \in I$ such that $U_i \to \Spec(R)$ admits a section. Let $S$ be a separated noetherian scheme, and let $\tau$ be a topology on the category $\Sch / S$ of finite type separated $S$-schemes.

\begin{defi} \label{defi:local}
Let us say that an $S$-scheme $P \to S$ (not subject to any finiteness conditions) is \emph{$(\Sch / S, \tau)$-local} if for every every $\tau$-cover $\{U_i \to X\}_{i \in I}$ in $\Sch / S$ the canonical morphism
\[ \amalg_{i \in I} \hom_S(P, U_i) \to \hom_S(P, X) \]
is surjective.
\end{defi}

Our first goal is to observe that for many nice topologies $\tau$ on $\Sch / S$, there is a canonical equivalence between the category of $(\Sch / S, \tau)$-local $S$-schemes whose structural morphism is affine, and the category of fibre functors on $\Shv_\tau(\Sch / S)$ (Theorem~\ref{theo:equi}). For this, we pass through a third category: the category of \emph{$\tau$-local pro-objects} (Definition~\ref{defi:proLocal}) in $\Sch / S$ (one could think of such pro-objects as the system of neighbourhoods of the ``point'' in question). One has the following equivalences of categories:
\begin{equation} \label{equa:equivalences}
\left \{ \begin{array}{cc} 
\textrm{  fibre functors } \\ 
\textrm{  on } \Shv_\tau(\C)
\end{array} \right  \}^{\mathrm{op}}
\cong 
\left \{ \begin{array}{cc} 
\tau\textrm{-local} \\ 
\textrm{pro-objects in } \C
\end{array} \right  \}
\cong
\left \{ \begin{array}{cc} 
(\Sch / S, \tau)\textrm{-local }  \\ 
\textrm{affine } S\textrm{-schemes} 
\end{array} \right  \}
\end{equation}

The first equivalence is an old topos theoretic result valid for any category $\C$ admitting finite limits and equipped with a topology $\tau$. The second equivalence is a standard application of the limit arguments for schemes in \cite[\S 8]{EGAIV3} applied to $\C = \Sch / S$, and is valid for any topology $\tau$ finer than the Zariski topology.%
\footnote{The second is valid in many other situations which the reader can work out according to their needs. For example, the étale, qfh, Nisnevich, Zariski, etc topologies admit various ``small'' sites, and when $\C$ is taken to be such a small site instead of $\Sch / S$, we obtain the full subcategory of the category of $(\C, \tau)$-local schemes consisting of those which are obtainable as inverse limits of pro-objects of $\C$.} %
The combination of this equivalence with Deligne's theorem leads to statements such as the following:

\begin{theo}[{cf. Theorem~\ref{theo:equi}}] \label{theo:detect}
Suppose that $S$ is a separated noetherian scheme, $\tau$ a topology on $\Sch / S$ finer than the Zariski topology, for which every covering family is refinable by a covering family indexed by a finite set. Then a morphism $f: F \to G$ in $\Shv_{\tau}(\Sch / S)$ is an isomorphism if and only if $f(P)$ is an isomorphism%
\footnote{By abuse of notation, by $f(P)$ we mean $\varinjlim_{(P \backslash \Sch / S)} f(X)$, where $(P \backslash  \Sch / S)$ is the category of factorizations $P \to X \to S$ with $X \in \Sch / S$.} %
 for every $(\Sch / S, \tau)$-local scheme $P$. 
\end{theo}







A second goal is to give an algebraic description of $(\Sch / S, \tau)$-local schemes for various topologies arising in algebraic geometry.

\begin{theoUn}[{\ref{theo:localSchemes}}]
Suppose that $S$ is a separated noetherian scheme. An affine $S$-scheme (i.e., an $S$-scheme whose structural morphism is affine) is $(\Sch / S, \tau)$-local if and only if it is ($\ast$) where $\tau$ and ($\ast$) are as in Table~\ref{tabl:local} on page \pageref{tabl:local}.
\end{theoUn}

\begin{table}
\begin{center}
\begin{tabular}{cc}
$\tau$ & ($\ast$) \\ \hline
Zariski & local \\
Nisnevich & local henselian \\
étale & local strictly henselian \\
closed & integral\\
cdf & integral and normal \\
finite & integral and \aic{} (Definition~\ref{defi:aic}) \\
qfh & local, integral, and \aic{} (Definition~\ref{defi:aic}) \\
rh & a valuation ring  \\
cdh & a henselian valuation ring \\
ldh & a henselian valuation ring whose fraction field has no non-trivial \\
&  finite extensions of degree prime to $l$ (cf. \cite[Def.\ 6.10]{Dat12}). \\
eh & a strictly henselian valuation ring \\
h & an \aic{} valuation ring (Definition~\ref{defi:aic}) \\ \hline
\end{tabular}
\caption{$(\Sch / S, \tau)$-local schemes. The words ``$\Spec$ of'' have been omitted from the last five.}
\label{tabl:local}
\end{center}
\end{table}

The descriptions for Zariski, Nisnevich, étale, rh, and h are already in the literature (we give references in the main text). The cases $\tau = $ cdh, ldh, and eh are immediate corollaries of these. The description for $\tau = $ finite appears to be new; the qfh case is an immediate consequence of this case. We remark that the cases qfh, rh, cdh, eh, and h of this theorem were observed by the %
first author 
 at Oberwolfach in August 2002 but this observation never appeared in print.

The reader will notice that the $\fppf$-topology is missing from Table~\ref{tabl:local}. We make some remarks about this in Section~\ref{sect:finite}. The reader can also consult \cite{Sch14} for some interesting results about schemes local for the $\fppf$-topology on the category of \emph{all} schemes (with no noetherian/quasi-compact/quasi-separated hypotheses) which overlap with some of our lemmas. Whereas our goal is to obtain an algebro-geometric description of a conservative family of fibre functors, the goal of \emph{ibid.} is to describe the topological space associated to the $\fppf$-topos of a scheme.

In Section~\ref{sect:directImage}, as an example of an application of Theorem~\ref{theo:equi} we show that the direct image between abelian sheaves along a closed immersion is exact for a number of topologies.

The second author's 
original motivation for thinking about Theorem~\ref{theo:equi} was the hope of finding a shortcut to the main result of \cite[Chapter 3]{Kel12}. In the end, due to non-noetherian rings being so much more complicated than noetherian ones, all %
he 
got was alternative proofs of some statements without any noticeable reduction in length (cf. Lemma~\ref{lemm:nonDiscrete} for the need to use non-noetherian rings).

\subsubsection*{Acknowledgements}

\thanks{ 
The second author 
was a post-doc at the University of Duisburg-Essen under Marc Levine's supervision when the majority of this article and its previous versions were written, and 
thanks %
him for the wonderful atmosphere in his work group, interesting conversations, and the support of the Alexander von Humboldt Foundation, and the DFG through the SFB Transregio 45.

The second author also thanks 
 Aise Johan de Jong for comments on a very early version, 
and the mathematicians who encouraged him to put this material in print despite the fact that much of it is already known to some.
}

\section{Points of locally coherent topoi}

In this section we recall Deligne's theorem (Theorem~\ref{theo:deligne}) on the existence of a conservative family of fibre functors. We also recall the equivalence between the category of fibre functors, and a subcategory of pro-objects of the underlying site (Proposition~\ref{prop:proPoints}). The material in this section is well-known to topos theorists.

Recall that a \emph{fibre functor} of a category $\cS$ is a functor $\cS \to \Set$ towards the category of sets which preserves finite limits and small colimits. By category of fibre functors of $\cS$, we mean the full subcategory of the category of functors from $\cS$ to $\Set$ whose objects are fibre functors. We will write $\Fib(\cS)$ for the category of fibre functors \cite[Def.\ IV.6.2]{SGA41}. If $\cS$ is a category of sheaves, the category of \emph{points} of $\cS$ is by definition $\Fib(\cS)^\textrm{op}$ \cite[Def.\ IV.6.1]{SGA41}.

Recall as well, that a \emph{pro-object} of a category $\C$ is a (covariant) functor $P_\bullet: \Lambda \to \C$ from a cofiltered%
\footnote{By \emph{cofiltered} we mean a small category $\Lambda$ that is non-empty, for every pair of objects $i, j$ there exists an object $k$ and morphisms $k \to i, k \to j$, and for every pair of parallel morphisms $i \rightrightarrows j$ there exists a morphism $k \to i$ such that the two compositions are equal \cite[Def.\ I.2.7]{SGA41}.} %
category. The pro-objects of a category $\C$ are the objects of a category $\Pro(\C)$ \cite[Equation I.8.10.5]{SGA41} where one defines
\[ \hom_{\Pro(\C)} \biggl ( (\Lambda \stackrel{P_\bullet}{\to} \C), (\Lambda' \stackrel{P_\bullet'}{\to} \C) \biggr ) = \varprojlim_{\lambda' \in \Lambda'} \varinjlim_{\lambda \in \Lambda} \hom_\C(P_\lambda, P_{\lambda'}'). \]
There is an obvious fully faithful functor $\C \to \Pro(\C)$ which sends an object $X \in \C$ to the constant pro-object $X: \ast \to \C$ with value $X$, where $\ast$ is the category with one morphism. Sometimes the category $\C$ is considered as a subcategory of $\Pro(\C)$ in this way.

Following Suslin and Voevodsky (and contrary to Artin and therefore Milne), we use the terms \emph{topology} and \emph{covering} as in \cite[Def.\ II.1.1]{SGA41} and \cite[Def.\ II.1.2]{SGA41} respectively. In particular, we have the following two properties.

\begin{lemm} \label{lemm:topologyConventions}
Suppose that $\C$ is equipped with a topology $\tau$. 
\begin{enumerate}
 \item \label{lemm:topologyConventions:refine} If a family of morphisms admits a refinement by a $\tau$-covering, then that family itself is also a covering \cite[Prop. II.1.4]{SGA41}.
 \item \label{lemm:topologyConventions:fibreDetects} Suppose $\C$ is essentially small and that $\Shv_\tau(\C)$ admits a conservative family of fibre functors $\{ \phi_j \}_{j \in J}$. Then a set of morphisms $\{p_i: U_i \to X \}_{i \in I}$ is a $\tau$-covering family if and only if for each $j \in J$ the family $\{\phi_j (p_i) \}_{i \in I}$ is a jointly surjective family of morphisms of sets \cite[Thm.\ II.4.4]{SGA41}.
\end{enumerate}
\end{lemm}

\begin{defi} \label{defi:proLocal}
Suppose that $\C$ is equipped with a topology $\tau$. Say that a pro-object $\Lambda \stackrel{P_\bullet}{\to} \C$ is \emph{$\tau$-local} if for every $X \in \C$ and every $\tau$-covering family $\{U_i \to X\}_{i \in I}$ the morphism 
\[ \amalg_{i \in I} \hom_{\Pro(\C)}(P_\lambda, U_i) \to \hom_{\Pro(\C)}(P_\lambda, X) \]
is surjective. Write $\Pro_\tau(\C)$ for the full subcategory of $\tau$-local pro-objects in $\Pro(\C)$.
\end{defi}

\begin{rema}
Rather than the surjectivity condition of Definition~\ref{defi:local} one is tempted to use a condition like ``every $\tau$-cover of $P$ admits a section''. We have avoided this because some topologies (such as the cdh and h for example) have multiple equivalent definitions, which are no longer equivalent when working with non-noetherian schemes (cf. \cite[Example 4.5]{GL}). As the schemes $P$ are often non-noetherian, we have chosen this statement to avoid the choice of non-noetherian versions of these topologies.
\end{rema}

\begin{prop}[{\cite[Prop. 7.13]{Joh77}}] \label{prop:proPoints}
Suppose that $\C$ is a small category that admits finite limits and is equipped with a topology $\tau$. Then the functor $\Pro(\C) \to \Fib(\PreShv(\C))^{op}$ which sends a pro-object $\Lambda \stackrel{P_\bullet}{\to} \C$ to the functor $F \mapsto \varinjlim_{\lambda \in \Lambda} F(P_\lambda)$ induces an equivalence of categories.
\begin{equation} \label{equa:equivOne}
\Pro_\tau(\C) \to \Fib(\Shv(\C))^{op}
\end{equation}
\end{prop}

\begin{rema}
In fact, the hypothesis that in Proposition~\ref{prop:proPoints} that $\C$ admits finite limits is not necessary.
\end{rema}

A quasi-inverse $\Fib(\Shv(\C))^{op} \to \Pro_\tau(\C)$ (i.e., an inverse up to natural isomorphism) is given as follows. For any functor $\Shv_\tau(\C) \stackrel{\phi}{\to} \Set$, let $(\ast \downarrow \phi)$ be the category whose objects are pairs $(X, s)$ with $X \in \C$, $s \in \phi(X)$ (where we identify $X$ with the $\tau$-sheafification of the presheaf it represents). The morphisms $(X, s) \to (Y, t)$ are those morphisms $f: X \to Y$ such that $\phi(f)(s) = t$. One can check that when $\phi$ is a fibre functor, $(\ast \downarrow \phi)$ is cofiltered, and therefore the canonical projection $(\ast \downarrow \phi) \to \C$ is a pro-object, and is in fact the $\tau$-local pro-object corresponding to $\phi$. Informative examples include the case where $(\C, \tau)$ is the site associated to a classical topological space, or the small étale (or Zariski) site of a noetherian scheme.


\begin{theo}[Deligne {\cite[Prop. VI.9.0]{SGA42} or \cite[Thm.\ 7.44, Cor.\ 7.17]{Joh77}}] \label{theo:deligne}
Suppose that $\C$ is a small category in which fibre products are representable, and is equipped with a topology $\tau$ such that every covering family of every object admits a finite subfamily which is still a covering family. 

Then a morphism $f$ in $\Shv_\tau(\C)$ is an isomorphism if and only if $\phi(f)$ is an isomorphism for every fibre functor $\phi \in \Fib(\Shv_\tau(\C))$. In fact, there exists a (proper) set of fibre functors which is still a conservative family.
\end{theo}

\begin{rema} \label{rema:quasiCompactObject}
The statement in \cite{Joh77} is actually less general, in that it assumes that $\C$ admits all finite limits. We have included the reference however because the proof is easier to read.
\end{rema}

\begin{rema}
The topologies in Definition~\ref{defi:topologies} satisfy the condition that every covering family is refinable by one with a finite index set. Hence, when $S$ is a separated noetherian scheme, Theorem~\ref{theo:deligne} applies to the category $\Sch / S$ equipped with any of these topologies.
\end{rema}

\section{Points of algebro-geometric categories of sheaves}

In this section we recall various topologies on the category of schemes of finite type over a separated noetherian base scheme, and give a geometric description the fibre functors for some of these sites.

Let $\Aff / S$ denote the category of affine $S$-schemes of finite type. That is, the category of $S$-schemes of finite type whose structural morphism is an affine morphism. The material in \cite[\S 8]{EGAIV3} allows us to replace pro-objects by honest schemes. The following is a direct consequence of the definitions and \cite[Prop. 8.13.5]{EGAIV3}.

\begin{prop} \label{prop:equiLimAff}
Suppose that $S$ is a separated noetherian scheme and $\alpha$ a topology on $\Aff / S$. The functor which sends a pro-object $\Lambda \stackrel{P_\bullet}{\to} \Aff / S$ to $\varprojlim_{\lambda \in \Lambda} P_\lambda$ induces an equivalence of categories 
\begin{equation} \label{equa:equiAff}
\Pro_{\alpha}(\Aff / S)
\cong
\left \{ \begin{array}{cc} 
(\Aff / S, \alpha)\textrm{-local }  \\ 
\textrm{affine } S\textrm{-schemes}
\end{array} \right  \}
\end{equation}
between $\Pro_{\alpha}(\Aff / S)$ (Definition~\ref{defi:proLocal}) and the category of $(\Aff / S, \alpha)$-local $S$-schemes with affine structural morphism (Definition~\ref{defi:local}).
\end{prop}

An explicit inverse is given as follows. For an $S$-scheme $P \to S$, define $(P \backslash \Sch / S)$ to be the category whose objects are factorizations $P \to X \to S$ with $X \in \Aff / S$ and morphisms are commutative diagrams $P {^\nearrow_\searrow} {\underset{X}{\stackrel{Y}{\downarrow}}} {^\searrow_\nearrow} S$. Since this category is co-filtered, projecting $P \to X \to S$ towards $X \to S$ gives a pro-object.

\begin{rema}
Note that the adjective ``affine'' is necessary if we want an equivalence of categories. For example, since $\hom_S(\Spec( f_* \OO_P ), X) = \hom_S(P, X)$ for any quasi-compact and quasi-separated $S$-scheme $P$ and any affine $S$-scheme $X$, the $S$-schemes $\Spec( f_* \OO_P )$ and $P$ determine the same pro-object of $\Aff / S$.
\end{rema}

Given a topology $\tau$ on $\Sch / S$ we will call the induced topology on $\Aff / S$ the affine $\tau$-topology or \emph{$\tau\aff$-topology}. That is, $\tau\aff$ is the finest topology on $\Aff / S$ such that the image in $\Sch / S$ of any $\tau\aff$-covering family is a $\tau$-covering family {\cite[Par. III.3.1, Cor. III.3.3]{SGA41}}.

\begin{theo} \label{theo:equi}
Suppose that $S$ is a separated noetherian scheme and $\tau$ a topology on $\Sch / S$ finer than the Zariski topology. The functor which sends a $(\Sch / S, \tau)$-local $S$-scheme $P \to S$ to the functor $\phi_P: F \mapsto \varinjlim_{(P \to X \to S)} F(X)$ (where the colimit is indexed by factorizations with $X \in \Sch / S$) induces an equivalence of categories:
\begin{equation} \label{equa:equi}
\left \{ \begin{array}{cc} 
(\Sch / S, \tau)\textrm{-local }  \\ 
\textrm{affine } S\textrm{-schemes}
\end{array} \right  \}
\cong
 \Fib \biggl ( \Shv_\tau(\Sch / S) \biggr ).
\end{equation}
If moreover, every covering family is refinable by one indexed by a finite set, then the family $\{ \phi_P \}$ indexed by $(\Sch / S, \tau)$-local $S$-schemes with affine structural morphism is a conservative family of fibre functors.
\end{theo}

\begin{rema}
It seems to be a non-trivial problem in general to show for a given topology $\tau$ on $\Sch / S$ that every covering family is refinable by one indexed by a finite set. For example, the Riemann-Zariski space \cite[\S 3]{GL} is used in \cite{GL} for this purpose for the h- and rh-topologies.
\end{rema}

\begin{proof}
If $\tau$ is a topology on $\Sch / S$ finer than the Zariski topology, then the canonical functor $\Shv_\tau(\Sch / S) \to \Shv_{\tau\aff}(\Aff / S)$ is an equivalence. %
Hence there is an equivalence of categories of fibre functors 
\[ \Fib(\Shv_\tau(\Sch / S)) \cong \Fib(\Shv_{\tau\aff}(\Aff / S)). \]
Now we have the equivalences
\[ \Fib(\Shv_{\tau\aff}(\Aff / S))^{op} \cong \Pro_{\tau\aff}(\Aff / S) \textrm{ and } \]
\[ \Pro_{\tau\aff}(\Aff / S) \cong 
\left \{ \begin{array}{cc} 
(\Aff / S, \tau\aff)\textrm{-local }  \\ 
\textrm{affine } S\textrm{-schemes} 
\end{array} \right  \}
\]
of Proposition~\ref{prop:proPoints} and Proposition~\ref{prop:equiLimAff}, and finally, we use again the fact that $\tau$ is finer than the Zariski topology to obtain the equivalence
\[ 
\left \{ \begin{array}{cc} 
(\Aff / S, \tau\aff)\textrm{-local }  \\ 
\textrm{affine } S\textrm{-schemes} 
\end{array} \right  \}
\cong 
\left \{ \begin{array}{cc} 
(\Sch / S, \tau)\textrm{-local }  \\ 
\textrm{affine } S\textrm{-schemes} 
\end{array} \right  \}. 
\]
The statement about a conservative family follows from Theorem~\ref{theo:deligne} of Deligne.
\end{proof}

Let us now consider specific topologies.

\begin{defi} \label{defi:topologies}
Let $S$ be a separated noetherian scheme. We consider the following topologies on $\Sch / S$. If $\sigma$ and $\rho$ are two topologies then we denote by $\langle \sigma, \rho \rangle$ the coarsest topology which is finer than both $\sigma$ and $\rho$.
\begin{enumerate}
 \item The \emph{Zariski topology (or $\Zar$)} is generated by finite families
\begin{equation} \label{equa:family}
\{ f_i: U_i \to X\}_{i = 1}^n
\end{equation}
which are jointly surjective (i.e., $\amalg U_i \to X$ is surjective on the underlying topological spaces) and such that each $U_i \to X$ is an open immersion. We allow all $n = 0$ so that the empty family is a covering family of the empty scheme.
 \item The \emph{Nisnevich topology (or $\Nis$)} is generated by completely decomposed families (\ref{equa:family}) with each $f_i$ an étale morphism. By \emph{completely decomposed} we mean that for each $x \in X$ there is an $i$ and a $u \in U_i$ such that $f_i(u) = x$ and $[k(u_i): k(x)] = 1$ (\cite{Nis89}).
\item The \emph{étale topology (or $\Et$)} is generated by families (\ref{equa:family}) which are jointly surjective and such that each $f_i$ is étale.
 \item The \emph{closed topology (or $\cl$)} is generated by families (\ref{equa:family}) which are jointly surjective and such that each $f_i$ is a closed immersion.
 \item Let $l$ be a prime integer. The \emph{finite-flat-surjective-prime-to-l topology (or $\fpsl$)} is generated by families $\{ Y \to X \}$ containing a single finite flat surjective morphism of constant degree prime to $l$, a prime integer.
 \item The \emph{finite-flat topology (or $\fps$)} is generated by families $\{ Y \to X \}$ containing a single finite flat surjective morphism.
 \item The \emph{completely decomposed finite topology (or $\cdf$)} is generated by completely decomposed families (\ref{equa:family}) with each $f_i$ a finite morphism.
 \item The \emph{envelope topology (or $\cdp$)} is generated by completely decomposed families (\ref{equa:family}) with each $f_i$ a proper morphism (\cite[Def.\ 18.3]{Ful}).
 \item The \emph{finite topology (or $\fin$)} is generated by families (\ref{equa:family}) which are jointly surjective and such that each $f_i$ is a finite morphism.
 \item The \emph{proper topology (or $\proptop$)} is generated by families (\ref{equa:family}) which are jointly surjective and such that each $f_i$ is a proper morphism.
 \item The \emph{$\rh$ topology} is rh $= \langle \Zar, \cdp \rangle$ (\cite[Def.\ 1.2]{GL}).
 \item The \emph{$\cdh$ topology} is $\cdh = \langle \Nis, \rh\rangle$ (\cite[Def.\ 5.7]{SV00}).
 \item The \emph{$\ldh$ topology} is $\ldh = \langle \cdh, \fpsl \rangle$ (\cite[Def.\ 3.2.1(2)]{Kel12}).
 \item The \emph{$\eh$ topology} is eh $= \langle \Et, \rh \rangle$ (\cite[Def.\ 2.1]{Gei06}).
 \item The \emph{$\fppf$ topology} is fppf $= \langle \Nis, \fps \rangle$ (\cite[\S IV.6.3]{SGA31}, cf. \cite[Cor.\ 17.16.2, Thm.\ 18.5.11(c)]{EGAIV4}).
 \item The \emph{$\qfh$ topology} is qfh $= \langle \Et, \fin \rangle$ (cf. \cite[Def.\ 3.1.2, Lem.\ 3.4.2]{Voe96}).
 \item The \emph{$\htop$ topology} is h $= \langle \Zar, \proptop \rangle$ (cf. \cite[Def.\ 3.1.2]{Voe96}, \cite[Def.\ 1.1, Thm.\ 4.1]{GL}).
\end{enumerate}
\end{defi}

The following diagram indicates some relationships. Not all relationships are shown.
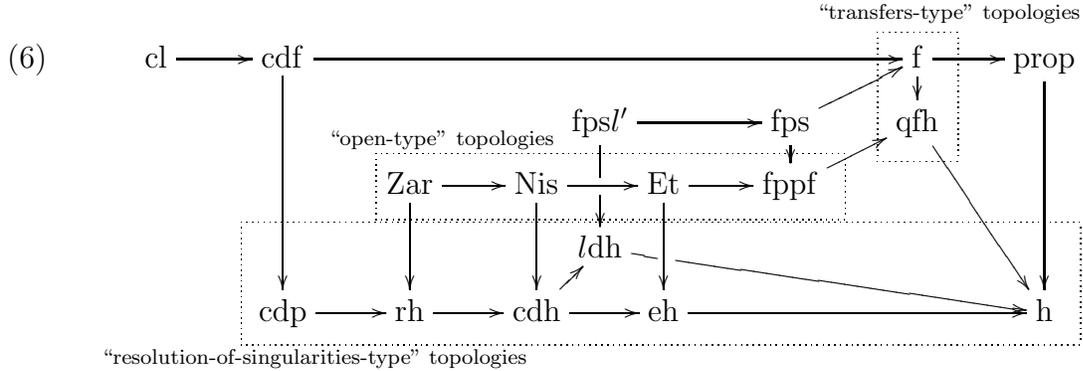
\begin{equation} \label{equa:dependancies}
\xymatrix@!=0pt{
\cl \ar[rr]  && \cdf \ar[rrrrrrrrrr] \ar[dddd] &&&&&&&&&& \fin \ar[rr] \ar[d]  && \proptop \ar[dddd] \\ 
&& && &&& \fpsl \ar[dd]|(0.5)\hole \ar[rrr] &&& \fps \ar[urr] \ar[d] && \qfh \ar[dddrr]  \\
&&&& \Zar \ar[rr] \ar[dd]  && \Nis \ar[rr] \ar[dd] && \Et \ar[rr] \ar[dd]  && \fppf \ar[urr] &&&& \\ 
&&{\phantom{fcdp}}&&&&& \ldh \ar[drrrrrrr]|(0.15)\hole \\
&&\cdp \ar[rr] && \rh \ar[rr] && \cdh \ar[rr] \ar[ur] && \eh \ar[rrrrrr] &&&&&& \htop
\save \ar@<2ex>@{}"1,13";"1,14"^{\textrm{``transfers-type'' topologies}} \restore
\save "1,13"."2,13"*+[F.]\frm{} \restore
\save \ar@<2ex>@{}"3,5";"3,6"^{\textrm{``open-type'' topologies}} \restore
\save "3,5"."3,11"*+[F.]\frm{} \restore
\save \ar@<-2ex>@{}"5,3";"5,4"_{\textrm{``resolution-of-singularities-type'' topologies}} \restore
\save "4,3"."5,15"*+[F.]\frm{} \restore
}
\end{equation}

Amongst the geometric descriptions of $(\Sch / S, \tau)$-local affine $S$-schemes which we give, the only one which is not either already in the literature, or follows from the others is $\tau = \fin$. This we postpone to the next section. 

For the proper and envelope topologies, not every covering family is refinable by one whose morphisms are affine. Consequently, for $\tau = \cdp$ and $\tau = \proptop$, the $(\Aff / S, \tau\aff)$-local $S$-schemes with affine structural morphism are not necessarily $(\Sch / S, \tau)$-local. Conversely, for $\tau = \cdf$ and $\tau = \fin$, an $S$-scheme with affine structural morphism is $(\Sch / S, \tau$)-local if and only if it is $(\Aff / S, \tau\aff$)-local (Lemma~\ref{lemm:fLocalSchemes}). This is also clearly true for any topology finer than (which includes equal to) the Zariski topology.

Our convention for the term \emph{valuation ring} is a ring $A$ which is an integral domain (i.e., a non-zero ring for which $ab {=} 0 \implies a {=} 0$ or $b {=} 0$) and such that for every $a \in \Frac(A)$, either $a \in A$ or $a^{-1} \in A$. We allow the totally ordered set of prime ideals of $A$ to have any order type.

\begin{theo} \label{theo:localSchemes}
Suppose that $S$ is a separated noetherian scheme. An $S$-scheme whose structural morphism is affine is $(\Sch / S, \tau)$-local if and only if it is ($\ast$), where $\tau$ and ($\ast$) are as in Table~\ref{tabl:local} on page \pageref{tabl:local}.
\end{theo}

\begin{rema}
The cases $\qfh, \rh, \cdh, \eh$, and $\htop$ in the above statement were observed by %
the first author 
at Oberwolfach in August 2002.
\end{rema}

\begin{proof}
In all cases, testing the local property on the empty covering of the empty scheme one sees that a $(\Sch / S, \tau)$-local scheme is non-empty.

The case $\tau = \Zar$ is observed in \cite[\S 2]{GL}.

For the case $\tau = \htop$ (resp. $\rh$), by the Zariski case we know that $P$ is affine (over $\Spec(\ZZ)$). But then the result is \cite[Prop. 2.2]{GL} (resp. \cite[Prop. 2.1]{GL}).

The cases $\tau = \Nis$ and $\tau = \Et$ are classical.

The cases $\tau = \fin$, $\tau = \cdf$, and $\tau = \cl$ are Lemma~\ref{lemm:fLocalSchemes} treated in the next section.

Finally, one notices that for any two topologies $\sigma, \rho$, an $S$-scheme is $(\Sch / S, \langle \sigma, \tau \rangle)$-local if and only if it is both $(\Sch / S, \sigma)$-local and $(\Sch / S, \tau)$-local. So the cases $\tau = \cdh, \ldh, \eh,$ and $\qfh$ follow from the others.
\end{proof}

\begin{lemm} \label{lemm:nonDiscrete}
Suppose that $S$ is a non-empty separated noetherian scheme and $\tau$ is one of $\rh, \cdh, \eh, \ldh,$ or $\htop$. %
%
If all valuations are required to be discrete (i.e., the value group is isomorphic to $\ZZ^n$ with the lexicographic order for some $n = 0, 1, 2, \dots$) 
then the family of fibre functors of $\Shv_\tau(\Sch / S)$
\[ \left \{  \phi_R: F \mapsto \varinjlim_{(\Spec(R) \to X \to S)} F(X)  \textrm{ such that } R \textrm{ is } (\ast) \textrm{ from Table~\ref{tabl:local} on page \pageref{tabl:local}} \right \} \]
is no longer conservative (cf. Theorem~\ref{theo:equi}).
\end{lemm}

\begin{proof}
By Lemma~\ref{lemm:topologyConventions}\eqref{lemm:topologyConventions:fibreDetects}, it suffices to construct a family $\{U_i \to X\}_{i \in I}$ of morphisms in $\Sch / S$ which is not a $\tau$-covering family but which induces an epimorphism $\amalg_{i \in I} U_i(R) \to X(R)$ for every discrete valuation ring $R$ which is ($\ast$).

For $\tau = \ldh$ or $\htop$, the family $\{ \GG_{m, S} \to \AA^1_S, S \stackrel{0}{\to} \AA^1_S\}$ works: If a valuation ring $R$ has fraction field which admits all $l'$-roots for some prime $l'$, then its value group has a canonical \mbox{$\ZZ[\tfrac{1}{l'}]$-module} structure. So if a discrete valuation ring has fraction field admitting all $l'$-roots for some prime $l'$ then it is a field. For any field $K$ over $S$ the morphism $\GG_{m}(K) \amalg \{ 0 \} \to \AA^1(K)$ is the isomorphism $K^* \amalg \{0\} \to K$. However, it is not possible to refine the morphism $\GG_{m, S} \amalg S \to \AA^1_S$ by a topological epimorphism (cf. \cite[Def. 3.1.1]{Voe96}) since $S$ is open in the source but not in the target. So the family $\{ \GG_{m, S} \to \AA^1_S, S \stackrel{0}{\to} \AA^1_S\}$ is not an $\htop$-covering family (cf. \cite[Def. 3.1.2]{Voe96}), and therefore not an $\ldh$-covering family either.

Now we describe a family which works for all $\tau$ in the statement. Let $s = \Spec(k)$ be a closed point of $S$. If $\alpha > 0$ is an irrational real number, one can consider the monomial valuation $v$ of $k(x, y)$ such that $v(x^iy^j) = i + j \alpha$. This valuation on a non-zero polynomial is the minimum of the valuations of the monomials occurring in it. This valuation dominates the local ring of $M = \Spec(k[x, y])$ at the origin, has residue field $k$, and defines a sequence of blowing-ups of $k$-points $M = M_0 \leftarrow M_1 \leftarrow M_2 \leftarrow \dots $ where $M_{n + 1}$ is the blow-up of $M_n$ at the center of $v$.

Let $M_n^\ast = M_n - \{$center of $v \}$ and let $R_v$ be the valuation ring of $v$. The family $\{ M_n^\ast \to M \}_{n \geq 0}$ is not an $\rh$-cover of $M$ since for all $n \geq 0$ the image of the unique factorisation $\Spec(R_v) \to M_n \to M$ of the canonical morphism $\Spec(R_v) \to M$ does not lie in $M_n^\ast$. However, we will show now that for any discrete valuation ring $V$ (or a valuation ring of rational rank $\leq 1$) the morphism $\amalg_{ n \geq 0} M_n^\ast(V) \to M(V)$ is surjective. Suppose we are given a morphism $f: \Spec(V) \to M$. Note that every point of $M$ lifts to a point of $M_1^*$ with trivial residue field extension so if $f$ factors through a point then $f$ lifts to $M_1^*$. Otherwise $f$ lifts uniquely to morphisms $\Spec(V) \to M_n$, and we claim that one of them must have image in $M_n^*$. Note that one has $R_v = \cup_{n \geq 0} \OO_{M_n, c_n}$ where $c_n \in M_n$ is the center of $v$ and the union takes place in the function field of $M$. So if no $\Spec(V) \to M_n$ has image inside $M_n^*$ then we get a local homomorphism of valuation rings $V \leftarrow R_v$, which must be injective as $R_v$ is of dimension one and the generic point of $\Spec(V)$ does not go to the origin of $M$, so it induces an order preserving inclusion of value groups which does not exist.

By considering the valuation ring of a choice of extension of $v$ to an algebraic closure of $k(x, y)$, one sees that $\{ M_n^\ast \to M \}_{n \geq 0}$ is also not an $h$-covering family and therefore not a $\cdh, \eh$, or $\ldh$ covering family either.  
\end{proof}

\begin{samepage}
\begin{rema}\ 
\begin{enumerate}
 \item In Definition~\ref{defi:topologies}, the Zariski, Nisnevich, étale, closed, $\cdp$, $\cdf$, finite, and proper topologies are defined in terms of finite families \eqref{equa:family} satisfying certain conditions. However, as our base scheme is noetherian, one can show that if an infinite family satisfies the condition in question, then some finite subfamily does as well.

 \item In any of the cases, by taking finite compositions (in any order) of the generating families one obtains a pretopology consisting of finite families which defines the topology.
\end{enumerate}
\end{rema}
\end{samepage}

\begin{rema}
Using absolute noetherian approximation, one can show that Theorem~\ref{theo:localSchemes} holds more generally for quasi-compact separated $S$-schemes. Under our definitions, for a non-empty $S$, there are non-separated $(\Sch / S, \tau)$-local schemes.
\end{rema}

\section{The finite, $\qfh$, and $\fppf$-topologies} \label{sect:finite}

This sections contains the commutative algebra required to characterise $(\Sch / S, \fin)$-local (and consequently $(\Sch / S, \qfh)$-local) schemes. We also make some basic comments on the $\fppf$-case.

\begin{defi} \label{defi:aic}
An integral ring $A$ is said to be \emph{absolutely integrally closed} or \aic{} if it is normal with algebraically closed fraction field. An integral scheme $X$ is said to be \emph{absolutely integrally closed} (or \aic{}) if all its local rings are \aic{}, or equivalently, if $A$ is \aic{} for every non-empty open affine $\Spec(A) \subseteq X$ (\cite[Lem.\ 5.8]{Dat12}).
\end{defi}

For various properties about absolutely integrally closed rings and schemes see \cite[\S 3, \S 4, and \S 5]{Dat12}. For example, suppose that $A$ is an \aic{} integral domain. Let $B \subseteq A$ be a subring which is integrally closed in $A$, let $\p \subseteq A$ be a prime ideal, and let $S \subseteq A$ be a multiplicatively closed subset not containing zero. Then $B$, $A / \p$ and $S^{-1}A$ are all \aic{} integral domains \cite[Lem.\ 3.5, 4.1, and 5.4]{Dat12}. Furthermore, an integral domain $A$ is \aic{} if and only if $A_\p$ (resp. $A_\m$) is \aic{} for all prime ideals $\p$ (resp. maximal ideals $\m$) of $A$ \cite[Lem.\ 5.8]{Dat12}.

\begin{lemm} \label{lemm:fLocalSchemes}
Let $S$ be a separated noetherian scheme and suppose that $P = \varprojlim P_\lambda$ is a projective limit of schemes $P_\lambda$ in $\Aff / S$. Let $\sigma = \fin, \cdf,$ or $\cl$ and let
\[ (\ast\ast) = \left \{ \begin{array}{ll}
\textrm{integral and \aic{}} & \textrm{ if } \sigma = \fin, \\ 
\textrm{integral and normal} & \textrm{ if } \sigma = \cdf, \\ 
\textrm{integral} & \textrm{ if } \sigma = \cl.
\end{array} \right . \]
The following are equivalent.
\begin{enumerate}
 \item \label{lemm:fLocalSchemes:aic} $P$ is $(\ast\ast)$.
 \item \label{lemm:fLocalSchemes:loc} $P$ is $(\Sch / S, \sigma)$-local.
 \item[{(2$\aff$\!\!)}] \label{lemm:fLocalSchemes:locAff} $P$ is $(\Aff / S, \sigma\aff)$-local.
\end{enumerate}
\end{lemm}

\begin{proof}
(\ref{lemm:fLocalSchemes:aic} $\Rightarrow$ \ref{lemm:fLocalSchemes:loc}). Suppose that $\{U_i \to X\}_{i = 1}^n$ is a $\sigma$-covering family of some $X \in \Sch / S$ and $P \to X$ is an $S$-morphism. It suffices to show that $P {\times_X} U_i \to P$ has a section for some $i \in I$. Choose an $i$ such that the inclusion of the generic point $\eta \to P$ factors through $P \times_X U_i \to P$ (if $\sigma = \cdf$ or $\cl$ then such an $i$ exists because $\amalg_{i = 1}^n U_i \to X$ is completely decomposed, and if $\sigma = \fin$ then such an $i$ exists because $\amalg_{i = 1}^n \eta {\times_X}U_i \to \eta$ is finite surjective and $\eta$ is the spectrum of an algebraically closed field). Let $V \subseteq P \times_X U_i$ be the closure of the image of $\eta$. Then $V \to P$ is a finite birational morphism of integral schemes with normal target, and is therefore an isomorphism.

(\ref{lemm:fLocalSchemes:loc} $\Rightarrow$ \ref{lemm:fLocalSchemes:locAff}$\aff$). Clear.

(\ref{lemm:fLocalSchemes:locAff}$\aff$ $\Rightarrow$ \ref{lemm:fLocalSchemes:aic}). Notice that we may replace each $P_\lambda$ by the scheme theoretic image of $P \to P_\lambda$, so that if $P$ is integral then the $P_\lambda$'s are integral and morphisms send generic points to generic points.

\emph{$P$ is integral.} Testing the local property on the empty covering of the empty scheme one sees that $P$ is non-empty, and so $P_\lambda$ is non-empty for all $\lambda$. Now testing the local property on the covering of each $P_\lambda$ by its irreducible components with the reduced subscheme structure, one sees that each $P_\lambda$ is integral, and hence $P$ is integral.

\emph{$P$ has algebraically closed function field if $\sigma = \fin$.} Let $R(P)$, $R(P_\lambda)$ denote the fields of rational fractions. If $E$ is a finite field extension of $R(P_\lambda)$ consider the normalization $\pi: Q \to P_\lambda$ of $P_\lambda$ in $E$ (\cite[\S 6.3]{EGAII}). By \cite[Tag 0817]{Sta}, $Q$ is a cofiltered limit of finite $P_\lambda$-schemes $Q_i$ corresponding to finite subalgebras of $\pi_*\OO_Q$. So $Q_i$ is integral and for some $i$, we have $R(Q_i) = E$. By the ``local'' property, $P \to P_\lambda$ lifts to $P \to Q_i$ so $R(Q_i)$ embeds in $R(P)$ over $R(P_\lambda)$. Now if $f(T) \in R(P)[T]$ is an irreducible polynomial over $R(P)$ then $f$ is contained in some $R(P_\lambda)[T]$ and defines a finite extension of $R(P_\lambda)$ so by the above, $f$ has a root in $R(P)$. Thus $R(P)$ is algebraically closed.

\emph{$P$ is normal if $\sigma = \fin$ or $\cdf$.} The normalization $P_\lambda^\sim \to P_\lambda$ is a cofiltered limit of finite birational $P_i \to P_\lambda$. By the ``local'' condition $P \to P_\lambda$ lifts to $P \to P_i$ for each $i$ (for $\sigma = \fin$ this is because each $\{P_i \to P_\lambda\}$ is a $\fin\aff$-covering family, and if $\sigma = \cdf$ this is because for each $i$ there exists a proper closed subscheme $V_i \to P_\lambda$ for which $\{V_i \to P_\lambda, P_i \to P_\lambda\}$ is a $\cdf\aff$-covering family and the dominant morphism $P \to P_\lambda$ cannot factor through the nondominant $V_i \to P_\lambda$). The lifting is unique as it is unique at the generic point and $P_i \to P_\lambda$ is separated. So for each $\lambda$ we get a lifting of $P \to P_\lambda$ to $P \to P_\lambda^\sim$, and these liftings are compatible with the morphisms $P_\mu^\sim \to P_\lambda^\sim$, and as $P^\sim = \varprojlim P_\lambda^\sim$ we get a section of the normalization morphism $P^\sim \to P$, so $P$ is normal.
\end{proof}

There is little we can say about $(\Sch / S, \fppf)$-local $S$-schemes. As we mentioned in the introduction, the reader can also consult (the independent) \cite{Sch14} for some interesting results about schemes local for the $\fppf$-topology on the category of \emph{all} schemes (with no noetherian/quasi-compact/quasi-separated hypotheses) which overlap with the following lemmas.

\begin{lemm} \label{lemm:fppfBasic}
Suppose $S$ is a separated noetherian scheme, and $\Spec(R) \to S$ is $(\Sch / S, \fppf)$-local.
\begin{enumerate}
 \item $R$ is strictly henselian.
 \item \label{enum:algClos} All residue fields of $R$ are algebraically closed.
 \item If $R$ is integral, then it is normal and therefore $(\Sch / S, \qfh)$-local.
 \item If $S \neq \varnothing$, there exists at least one $(\Sch / S, \fppf)$-local $\Spec(R) \to S$ such that $R$ is not integral.
 \item If $R$ is noetherian, then it is reduced. Consequently, condition (\ref{enum:algClos}) above is not sufficient for $\Spec(R) \to S$ to be $(\Sch / S, \fppf)$-local.
\end{enumerate}
\end{lemm}

\begin{proof}\ 
\begin{enumerate}
 \item This is an immediate consequence of the $\fppf$-topology being finer than the étale topology.

 \item Suppose that $\p$ is a prime ideal of $R$, let $\kappa = \Frac(R / \p)$ and suppose that $f(T) = T^n + \sum_{i = 0}^{n - 1} \tfrac{b_i}{c_i} T^i \in \kappa[T]$ is a monic polynomial. We will confuse the $b_i, c_i \in R$ with their images in $R / \p$. Let $c = \prod_{i = 0}^{n - 1} c_i$, let $a_i =  c^{n - i}\tfrac{b_i}{c_i} \in R$ (for $0 \leq i \leq n - 1$), and set \mbox{$g(T) = T^n + \sum_{i = 0}^{n - 1} a_i T^i$.} Now $\Spec(R[T] / (g(T))) \to \Spec(R)$ is a finite flat surjective morphism. The polynomial $g$ therefore has some solution $d \in R$ since $\Spec(R)$ is $(\Sch / S, \fppf)$-local.%
\footnote{Strictly speaking, one should use limit arguments to lift $\Spec(R[T] / (g(T))) \to \Spec(R)$ to a finite flat surjective morphism $U \to X$ in $\Aff / S$ equipped with an $S$-morphism $\Spec(R) \to X$, and then convert a factorization $\Spec(R) \to U \to X$ into a solution $d \in R$. For example, let $\mathcal{A}$ be the $\OO_S$-algebra generated by the $a_i$ and take $X = \Spec(\mathcal{A})$.} %
Notice that $c$ is not zero in $R / \p$ since the $c_i$ are not zero in $R / \p$, so $\tfrac{d}{c}$ is a well-defined element of $\kappa = \Frac(R / \p)$. Now in $\kappa$ we have $f(\tfrac{d}{c}) = \tfrac{1}{c^n} g(d) = 0$.


 \item Notice that if $R$ is an integral ring for which every finite flat $R$-algebra admits a retraction, then every monic in $R[T]$ splits into linear factors. Consequently, $R$ is integrally closed in its field of fractions. For the $\qfh$-statement, since $\qfh = \langle \Et, \fin \rangle$, and $\Et$ is coarser than $\fppf$, it suffices to show that $\Spec(R) \to S$ is $(\Sch / S, \fin)$-local. But $R$ is normal with algebraically closed fraction field, so $\Spec(R)$ is \aic{}, and therefore $(\Sch / S, \fin)$-local (Lemma~\ref{lemm:fLocalSchemes}).

 \item Suppose the contrary. By the previous part, this would then imply that the class of $(\Sch / S, \fppf)$-local affine $S$-schemes and the class of $(\Sch / S, \qfh)$-local affine $S$-schemes are the same. But this would then imply that the $\qfh$ and $\fppf$-topologies were equal (Lemma~\ref{lemm:topologyConventions}(\ref{lemm:topologyConventions:fibreDetects})). This is false since there are many surjective finite morphisms in $\Sch / S$ which are not refinable by flat ones.

 \item Let $I = \{r : r^n = 0 $ for some $n > 0 \}$ be the nilradical. Let $n$ be a positive integer $n$ such that $I^n = 0$ (existence of such an $n$ is where we use the hypothesis that $R$ is noetherian). For every $r \in I$, the $R$-algebra \mbox{$R[T] / (T^n - r)$} is finite and flat, and therefore admits a retraction. Equivalently, there exists $s \in R$ such that $s^n = r$. But then $(s^n)^n = 0$ so $s \in I$, and therefore $s^n = 0$, and hence $r = 0$. So $I = \{0\}$. \qedhere
\end{enumerate}
\end{proof}

\begin{rema}[{\cite[\href{http://stacks.math.columbia.edu/tag/04C5}{Tag 04C5}]{Sta}}]
Heuristically, for a ring $A$, an ideal $I$, and a flat, finitely presented algebra $A / I \to B_0$, there is no reason to suppose the existence of a flat, finitely presented $A$-algebra $A \to B$ for which $B_0 = B \otimes_A (A / I)$. So one might conjecture that Proposition~\ref{prop:directImageExactAll} is false for the $\fppf$-topology. If one believes such a conjecture, then one has some further information about $(\Sch / X, \fppf)$-local schemes.
\end{rema}

\begin{lemm}
Suppose that there exists a closed immersion of separated noetherian schemes \mbox{$i: Z \to X$} such that the direct image $i_*: \Shv_\fppf(\Sch / Z, \Ab) \to \Shv_\fppf(\Sch / X, \Ab)$ on abelian $\fppf$-sheaves is NOT exact. Then there exists an $(\Sch / X, \fppf)$-local scheme $P \to X$ with affine structural morphism, and a closed non-empty subscheme $Q \subseteq P$ such that $Q$ is not $(\Sch / X, \fppf)$-local.
\end{lemm}

\begin{proof}
The converse contradicts the proof of Proposition~\ref{prop:directImageExactAll}. More explicitly, if $F \to F'$ is an epimorphism of abelian $\fppf$-sheaves on $Z$ such that $i_*F \to i_*F'$ is not an epimorphism, let $P \to X$ correspond to a fibre functor $\phi$ on $\Shv_\fppf(\Sch / X, \Ab)$ under which $\phi(i_*F) \to \phi(i_*F')$ is not surjective and $Q = Z {\times_X} P$.
\end{proof}

\section{Exactness of direct image} \label{sect:directImage}

In this section we show that direct image along a closed immersion is exact for various topologies.

Recall that a functor $u: \C \to \C'$ between small categories $\C, \C'$ equipped with topologies $\tau, \tau'$ respectively is \emph{continuous} if for every $\tau'$-sheaf $F$ the composition $F \circ u$ is a $\tau$-sheaf \cite[Def.\ III.1.1]{SGA41}. If fibre products are representable in $\C$ and $u$ commutes with them, then $u$ is continuous if and only if for every $\tau$-covering family \mbox{$\{U_i \to X\}_{i \in I}$} in $\C$, the family \mbox{$\{u(U_i) \to u(X)\}_{i \in I}$} is a $\tau'$-covering family \cite[Prop. III.1.6]{SGA41}.

We would like to use the notion of a \emph{co}continuous morphism of sites given in \cite[Def.\ III.2.1]{SGA41} (see also \cite[Def.\ II.1.2]{SGA41}). However, closed immersions with non empty complement do not give rise to a cocontinuous morphism for the usual topologies, see \cite[\href{http://stacks.math.columbia.edu/tag/00XV}{Tag 00XV}]{Sta}.

\begin{defi}[{\cite[\href{http://stacks.math.columbia.edu/tag/04B4}{Tag 04B4}]{Sta}}] \label{defi:almostCo}
A functor $u: \C \to \C'$ between small categories $\C, \C'$ equipped with topologies $\tau, \tau'$ respectively is \emph{almost cocontinuous} if for every $\tau'$-covering family of the form  \mbox{$\mathcal{U} = \{U_i \to u(X)\}_{i \in I}$} there exists a $\tau$-covering family $\mathcal{V} = \{V_j \to X\}_{j \in J}$ such that either
\begin{enumerate}
 \item the image of $\mathcal{V}$ under $u$ is a refinement of $\mathcal{U}$, or
 \item for each $j$, the empty family is a covering of $u(V_j)$ in $\C'$.
\end{enumerate}
\end{defi}

\begin{rema}
The stacks project calls the functor $u$ in the above definition \mbox{\emph{cocontinuous}} if for every $\mathcal{U}$ there exists a $\mathcal{V}$ such that the first condition is satisfied (\cite[\href{http://stacks.math.columbia.edu/tag/00XI}{Tag 00XI}]{Sta}).
\end{rema}

\begin{prop} \label{prop:contLimites}
Suppose that $u: \C \to \C'$ is a functor between categories small $\C, \C'$ equipped with topologies $\tau, \tau'$.
\begin{enumerate}
 \item \label{prop:contLimites:one}If $u$ is continuous then the functor $- \circ u$ preserves sheaves, and therefore induces a functor $u_s: \Shv_\tau(\C') \to \Shv_\tau(\C)$ \cite[Paragraph III.1.11]{SGA41}. Furthermore, if $\C$ is small, then $u_s$ admits a left adjoint $u^s$, and therefore $u_s$ preserves all small limits \cite[Prop. 1.3(1)]{SGA41}.
 \item For the sake of rigour,%
\footnote{The Stacks Project uses slightly less general definitions of ``site'', ``topology'', ``continuous'' than SGA. However, if the underlying categories are small, all fibre products are representable, and $u$ preserves fibre products, then the SGA meaning of these terms agrees with The Stacks Project's meaning.} %
 suppose that fibre products are representable in $\C$ and $u$ preserves them. If $u$ is continuous and almost cocontinuous, then $u_s$ preserves finite connected colimits. {\cite[\href{http://stacks.math.columbia.edu/tag/04B9}{Tag 04B9}]{Sta}}
\end{enumerate}
\end{prop}

\begin{warn}
We follow the notation $u^s, u_s$ from \cite{SGA41}. Unfortunately \cite{Sta} writes $u^s$ for the $u_s$ of \cite{SGA41}, and $u_s$ for the $u^s$ of \cite{SGA41}. So the adjunction $(u^s, u_s)$ of \cite{SGA41} is written in \cite{Sta} as $(u_s, u^s)$.
\end{warn}

\begin{prop} \label{prop:directImageExactAll}
Suppose that $i: Z \to X$ is a closed immersion of separated noetherian schemes. Then the direct image
\[ i_*: \Shv_\tau(\Sch / Z, \Ab) \to \Shv_\tau(\Sch / X, \Ab) \]
is exact if $\tau$ is any of Zar, Nis, ét, rh, cdh, ldh, eh, or h.
\end{prop}

\begin{proof}
The functor $Z \times_X -: \Sch / X \to \Sch / Z$ is a continuous morphism of sites and so $i_*$ has a left adjoint and therefore preserves all small limits (Proposition~\ref{prop:contLimites}(\ref{prop:contLimites:one})).

If $\tau = \Zar, \Nis$, or $\Et$ we claim that the functor $Z {\times_X} -: \Sch / X \to \Sch / Z$ is also an almost cocontinuous morphism of sites. Let $Y \in \Sch / X$ and consider a $\tau$-covering family $\mathcal{U} = \{U_i \to Z {\times_X} Y\}_{i \in I}$ in $\Sch / Z$. To prove the claim, we must find a $\tau$-covering family in $\mathcal{V} = \{V_j \to Y\}_{j \in J}$ that satisfies one of the two conditions in Definition~\ref{defi:almostCo}.

If $Z {\times_X} Y$ is the empty scheme, then $\mathcal{V} = \{ Y \stackrel{\id}{\to} Y \}$ satisfies the second condition. If $Z {\times_X} Y$ is non-empty, then we will construct a family \mbox{$\mathcal{V} = \{p_j : V_j \to Y \}_{j \in J}$} such that for every $(\Sch / X, \tau)$-local affine $X$-scheme the morphism $\amalg_{j \in J} \hom_X(P, V_j) \to \hom_X(P, Y)$ it induces is surjective, and the image of $\mathcal{V}$ under $Z {\times_X}-$ refines $\mathcal{U}$. Since every fibre functor of $\Sch_\tau(\Sch / X)$ is induced by a $(\Sch / X, \tau)$-local affine (Theorem~\ref{theo:equi}) it then follows that for every fibre functor $\phi$, the family $\{ \phi(p_j) \}_{j \in J}$ is jointly surjective and therefore $\mathcal{V}$ is a covering family (Lemma~\ref{lemm:topologyConventions}(\ref{lemm:topologyConventions:fibreDetects})), and the second condition in Definition~\ref{defi:almostCo} is satisfied.

For each $(\Sch / X, \tau)$-local affine $X$-scheme $P \to X$, by Theorem~\ref{theo:localSchemes} the $Z$-scheme $Z {\times_X} P \to Z$ is either empty or $(\Sch / Z, \tau)$-local and affine. Therefore, as $I$ is non-empty, for each $X$-morphism \mbox{$j: P \to Y$} there is an $i$ such that $Z {\times_X} P \to Z {\times_X} Y$ factors through the morphism $U_i \to Z {\times_X} Y$. If one presents $P$ as the inverse limit of a pro-object $(P_\lambda)_{\lambda \in \Lambda}$ of $\Aff / Y$ (using our chosen morphism $P \to Y$), then we have $Z {\times_X} P \cong Z {\times_X} (\varprojlim_\Lambda P_\lambda) \cong \varprojlim_\Lambda (Z {\times_X} P_\lambda)$ and the standard limit arguments (\cite[Prop. 8.13.1]{EGAIV3}) provide a $\lambda$ and a factorization $Z {\times_X} P \to Z {\times_X} P_\lambda \to U_i \to Z {\times_X} Y$. That is, we have a $V_P \to Y$ in $\Sch / X$ (take $V_P = P_\lambda$) equipped with factorizations $P \to V_P \to Y$ and $Z {\times_X} (V_P) \to U_i \to Z {\times_X} Y$. The family $\{V_j \to Y \}_{\{P \stackrel{j}{\to} Y\}}$ indexed by the class $J$ of morphisms with source a $(\Sch / X, \tau)$-local affine and target $Y$ satisfies our requirements.

If $\tau = $ rh, cdh, ldh, eh, or h, then the same proof works with slight modifications. We replace $\Shv_\tau(\Sch / Z)$ with the equivalent category $\Shv_{\tau_\red}(\Sch_\red / Z_{\red})$ where $\Sch_\red / Z_{\red}$ is the category of \emph{reduced} separated schemes of finite type over $Z_\red$ equipped with the topology $\tau_\red$ induced from the inclusion $\Sch_\red(Z_\red) \subseteq \Sch / Z$. We must also replace $Z {\times_X} P$ with $(Z {\times_X} P)_{\red}$. Then the above proof works.
\end{proof}


\begin{rema}
The above proof does not work with $\qfh$ because even though every integral closed subscheme of a $(\Sch / S, \qfh)$-local scheme with affine structural morphism is $(\Sch / S, \qfh)$-local (\cite[Lem.\ 4.1]{Dat12}), this is not true for reducible reduced closed subschemes.
\end{rema}


\begin{thebibliography}{SGA72b}

\bibitem[Dat12]{Dat12}
R.~Datta.
\newblock {Polygons in Quadratically Closed Rings and Properties of n-adically
  Closed Rings}.
\newblock {\em Arxiv preprint arXiv:1209.2220}, 2012.

\bibitem[Ful98]{Ful}
W.~Fulton.
\newblock {\em Intersection theory}, volume~2 of {\em Ergebnisse der Mathematik
  und ihrer Grenzgebiete. 3. Folge. A Series of Modern Surveys in Mathematics
  [Results in Mathematics and Related Areas. 3rd Series. A Series of Modern
  Surveys in Mathematics]}.
\newblock Springer-Verlag, Berlin, second edition, 1998.

\bibitem[Gei06]{Gei06}
T.~Geisser.
\newblock Arithmetic cohomology over finite fields and special values of
  {$\zeta$}-functions.
\newblock {\em Duke Math. J.}, 133(1):27--57, 2006.

\bibitem[GL01]{GL}
T.~G. Goodwillie and S.~Lichtenbaum.
\newblock A cohomological bound for the {$h$}-topology.
\newblock {\em Amer. J. Math.}, 123(3):425--443, 2001.

\bibitem[EGAII]{EGAII}
A.~Grothendieck.
\newblock \'{E}l\'ements de g\'eom\'etrie alg\'ebrique (r\'edig\'es avec la collaboration de Jean Dieudonn\'e) {II}. \'{E}tude globale \'el\'ementaire de quelques classes de morphismes.
\newblock {\em Inst. Hautes \'Etudes Sci. Publ. Math.}, (8):222, 1961.

\bibitem[EGAIV3]{EGAIV3}
A.~Grothendieck.
\newblock \'{E}l\'ements de g\'eom\'etrie alg\'ebrique (r\'edig\'es avec la collaboration de Jean Dieudonn\'e) {IV}. \'{E}tude locale
  des sch\'emas et des morphismes de sch\'emas, Troisi\`eme partie.
\newblock {\em Inst. Hautes \'Etudes Sci. Publ. Math.}, (28):255, 1966.

\bibitem[EGAIV4]{EGAIV4}
A.~Grothendieck.
\newblock \'{E}l\'ements de g\'eom\'etrie alg\'ebrique (r\'edig\'es avec la collaboration de Jean Dieudonn\'e) {IV}. \'{E}tude locale
  des sch\'emas et des morphismes de sch\'emas, Quatri\`eme partie.
\newblock {\em Inst. Hautes \'Etudes Sci. Publ. Math.}, (32):361, 1967.

\bibitem[Joh77]{Joh77}
P.~T. Johnstone.
\newblock {\em Topos theory}.
\newblock Academic Press [Harcourt Brace Jovanovich Publishers], London, 1977.
\newblock London Mathematical Society Monographs, Vol. 10.

\bibitem[Kel12]{Kel12}
S.~Kelly.
\newblock {\em Triangulated categories of motives in positive characteristic}.
\newblock PhD thesis, Universit{\'e} Paris 13, Australian National University,
  2012.

\bibitem[Nis89]{Nis89}
Y. Nisnevich.
\newblock The completely decomposed topology on schemes and associated descent
  spectral sequences in algebraic {$K$}-theory.
\newblock In {\em Algebraic {$K$}-theory: connections with geometry and
  topology ({L}ake {L}ouise, {AB}, 1987)}, volume 279 of {\em NATO Adv. Sci.
  Inst. Ser. C Math. Phys. Sci.}, pages 241--342. Kluwer Acad. Publ.,
  Dordrecht, 1989.

\bibitem[Sch14]{Sch14}
S.~Schr{\"o}er.
\newblock {Points in the fppf topology, preprint}.
\newblock {\em Arxiv preprint arXiv:1407.5446}, 2014.

\bibitem[SGA70]{SGA31}
{\em Sch{\'e}mas en groupes. I: Propri{\'e}t{\'e}s g{\'e}n{\'e}rales des
  sch{\'e}mas en groupes}.
\newblock Lecture Notes in Mathematics, Vol. 151. Springer-Verlag, Berlin,
  1970.
\newblock S{\'e}minaire de G{\'e}om{\'e}trie Alg{\'e}brique du Bois-Marie
  1962/1964 (SGA 3), Dirig{\'e} par M. Demazure et A. Grothendieck.

\bibitem[SGA72a]{SGA41}
{\em Th\'eorie des topos et cohomologie \'etale des sch\'emas. {T}ome 1:
  {T}h\'eorie des topos}.
\newblock Lecture Notes in Mathematics, Vol. 269. Springer-Verlag, Berlin,
  1972.
\newblock S{\'e}minaire de G{\'e}om{\'e}trie Alg{\'e}brique du Bois-Marie
  1963--1964 (SGA 4), Dirig{\'e} par M. Artin, A. Grothendieck, et J. L.
  Verdier. Avec la collaboration de N. Bourbaki, P. Deligne et B. Saint-Donat.

\bibitem[SGA72b]{SGA42}
{\em Th\'eorie des topos et cohomologie \'etale des sch\'emas. {T}ome 2}.
\newblock Lecture Notes in Mathematics, Vol. 270. Springer-Verlag, Berlin,
  1972.
\newblock S{\'e}minaire de G{\'e}om{\'e}trie Alg{\'e}brique du Bois-Marie
  1963--1964 (SGA 4), Dirig{\'e} par M. Artin, A. Grothendieck et J. L.
  Verdier. Avec la collaboration de N. Bourbaki, P. Deligne et B. Saint-Donat.

\bibitem[{Sta}14]{Sta}
The {Stacks Project Authors}.
\newblock {\itshape Stacks Project}.
\newblock \url{http://stacks.math.columbia.edu}, 2014.

\bibitem[SV00]{SV00}
A.~Suslin and V.~Voevodsky.
\newblock Bloch-{K}ato conjecture and motivic cohomology with finite
  coefficients.
\newblock In {\em The arithmetic and geometry of algebraic cycles ({B}anff,
  {AB}, 1998)}, volume 548 of {\em NATO Sci. Ser. C Math. Phys. Sci.}, pages
  117--189. Kluwer Acad. Publ., Dordrecht, 2000.

\bibitem[Voe96]{Voe96}
V.~Voevodsky.
\newblock Homology of schemes.
\newblock {\em Selecta Math. (N.S.)}, 2(1):111--153, 1996.

\end{thebibliography}

\end{document}